\documentclass[a4paper,reqno,11pt]{amsart}
\usepackage[T1]{fontenc}
\usepackage[utf8]{inputenc}
\usepackage{amssymb, amsmath, amsthm, graphicx, enumerate, enumitem, color}
\usepackage[dvipsnames,svgnames,table]{xcolor}
\usepackage[foot]{amsaddr}
\usepackage[normalem]{ulem}

\makeatletter
\def\thm@space@setup{
\thm@preskip=4mm
\thm@postskip=0mm
}
\makeatother

\usepackage{hyperref}
\usepackage{doi}
\usepackage[capitalise, compress, noabbrev]{cleveref}
\definecolor{linkblue}{named}{MidnightBlue}
\hypersetup{colorlinks=true, linkcolor=linkblue,  anchorcolor=linkblue,
	citecolor=linkblue, filecolor=linkblue, menucolor=linkblue,
	urlcolor=linkblue,
    pdftitle={Tree decompositions whose trees are subgraphs: An application of Simon's factorization}
    } 
\usepackage{mathtools}
\usepackage{thmtools, thm-restate}
\usepackage[longnamesfirst,numbers,sort&compress]{natbib}

\usepackage{textpos}

\theoremstyle{plain}
\newtheorem{thm}{Theorem}
\newtheorem*{thm*}{Theorem}
\newtheorem{theorem}[thm]{Theorem}

\newtheorem{lemma}[thm]{Lemma}
\newtheorem*{lemma*}{Lemma}

\newtheorem*{cor*}{Corollary}

\newtheorem{claim}[thm]{Claim}
\newtheorem*{lem*}{Lemma}

\newtheorem*{conjecture*}{Conjecture}

\newcommand{\N}{\mathbb{N}}
\let\le\leqslant

\let\leq\leqslant
\let\geq\geqslant

\let\epsilon\varepsilon

\let\setminus-
\let\varphi\phi

\DeclareMathOperator\pw{pw}
\DeclareMathOperator\tw{tw}

\DeclareMathOperator\hrank{height}

\DeclareMathOperator\torso{torso}
\DeclareMathOperator\cmp{cmp}

\newcommand{\abst}[1]{[\![#1]\!]}

\DeclarePairedDelimiter\set{\{}{\}}

\setenumerate{label=\textup{(\roman*)}, noitemsep, topsep=0pt,
labelindent=.2em, leftmargin=*, widest=iii,}
\setitemize{noitemsep, topsep=-\parskip, labelindent=.2em, leftmargin=*, widest=iii,}

\title[Tree decompositions whose trees are subgraphs]{Tree decompositions whose trees are subgraphs:\\ An application of Simon's factorization}

\begin{document}

\author[Bourneuf]{Romain Bourneuf}
\address[R.~Bourneuf]{LaBRI, Université de Bordeaux, Bordeaux, France}
\email{romain.bourneuf@ens-lyon.fr}

\author[Joret]{Gwena\"el Joret}
\address[G.~Joret]{D\'epartement d'Informatique, Universit\'e libre de Bruxelles, Belgium}
\email{gwenael.joret@ulb.be}

\author[Micek]{Piotr Micek}
\address[P.~Micek]{Department of Theoretical Computer Science, Jagiellonian University, Kraków, Poland}
\email{piotr.micek@uj.edu.pl}

\author[Milani\v{c}]{Martin Milani\v{c}}
\address[M.~Milani\v{c}]{Faculty of Mathematics, Natural Sciences and Information Technologies and Andrej Marušič Institute, University of Primorska, Koper, Slovenia}
\email{martin.milanic@upr.si}

\author[Pilipczuk]{Michał Pilipczuk}
\address[M.~Pilipczuk]{Institute of Informatics, Faculty of Mathematics, Informatics, and Mechanics, University of Warsaw, Poland}
\email{michal.pilipczuk@mimuw.edu.pl}

\thanks{G.\ Joret is supported by the Belgian National Fund for Scientific Research (FNRS). P.\ Micek is supported by the National Science Center of Poland under grant UMO-2023/05/Y/ST6/00079 within the WEAVE-UNISONO program. 
M.\ Milani\v{c} is supported in part by the Slovenian Research and Innovation Agency (I0-0035, research program P1-0285 and research projects J1-60012, J1-70035, J1-70046, and N1-0370) and by the research program CogniCom (0013103) at the University of Primorska.
M.\ Pilipczuk is supported by the project BOBR that has received funding from the European Research Council (ERC) under the European Union’s Horizon 2020 research and innovation programme (grant agreement No. 948057).}

\begin{abstract}
	We show that every connected graph $G$ has a tree decomposition indexed by a tree $T$ such that $T$ is a subgraph of $G$ and the width of the tree decomposition is bounded from above by a function of the pathwidth of $G$.
	This answers a question of  Blanco, Cook, Hatzel, Hilaire, Illingworth, and
	McCarty (2024), who proved that it is not possible to have such a tree decomposition whose width is bounded by a function of the treewidth of $G$.

	The proof relies on Simon's Factorization Theorem for finite semigroups, a tool that has already been applied successfully in various areas of graph theory and combinatorics in recent years. Our application is particularly simple and can serve as a good introduction to this technique.
\end{abstract}

\maketitle
\begin{textblock}{20}(-1.4, 6.6)
	\includegraphics[width=40px]{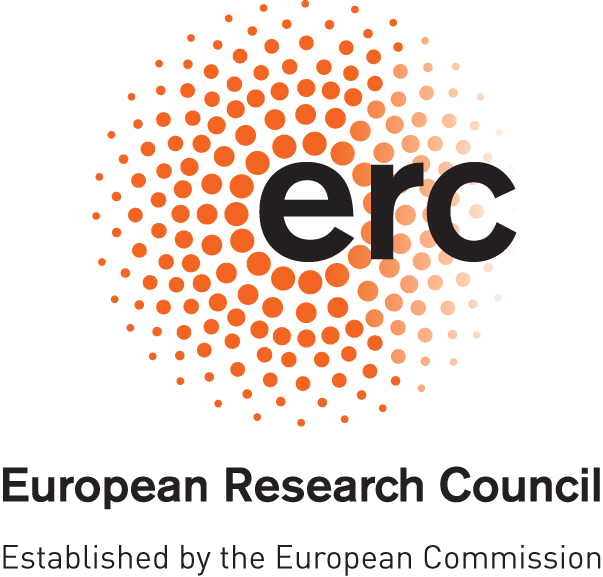}%
\end{textblock}
\begin{textblock}{20}(-1.4, 7.5)
	\includegraphics[width=40px]{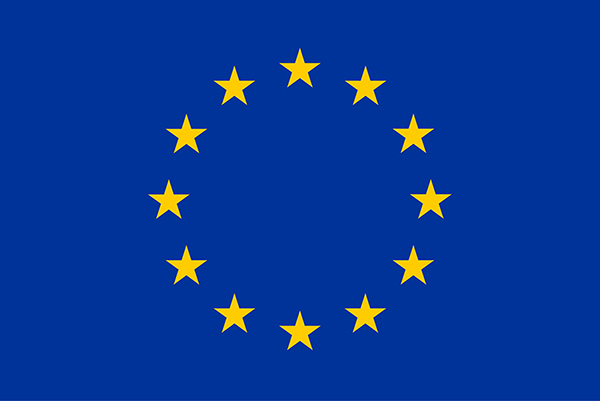}%
\end{textblock}

\section{Introduction}

Graphs admitting tree decompositions of small width are one of the main building blocks of structural and algorithmic graph theory. Informally speaking, a tree decomposition displays how the considered graph can be pieced together from smaller parts combined along a tree, which we shall call the \emph{indexing tree}. A natural expectation is that the geometry of the indexing tree should roughly follow the geometry of the graph itself.
\citet{D19} made this question concrete by asking the following: Is it true that every connected graph $G$ has a tree decomposition whose indexing tree is a subgraph of $G$, and whose width is bounded by some function of the treewidth of $G$?

Dvo\v{r}\'ak's question has recently been answered in the negative by \citet{BCHHIM2024}.
In fact, they showed that this is not possible even if we only require the indexing tree to be a minor of $G$.

The construction in~\cite{BCHHIM2024} has bounded treewidth (treewidth $2$ in fact) but has unbounded pathwidth.
This observation lead the authors of \cite{BCHHIM2024} to relax Dvořák's question as follows: Is it true that every connected graph $G$ has a tree decomposition whose indexing tree is a subgraph of~$G$, and whose width is bounded by some function of the {\em pathwidth} of $G$?
In this paper, we answer this question positively.

\begin{restatable}{theorem}{restateMain}\label{thm:main_intro}
There exists a function $f:\N \to \N$ such that, for every connected graph $G$ of pathwidth less than $k$, there is a tree decomposition of $G$ of width at most $f(k)$ indexed by a tree $T$ that is a spanning tree of $G$. Moreover, in this tree decomposition, every vertex belongs to its own bag.
\end{restatable}

It is worth noting that in general we cannot hope to obtain a path decomposition of $G$ of width at most $f(k)$ indexed by a path $P$ that is a subgraph of $G$: the family of stars is a counterexample.

The main tool in our proof is Simon's Factorization Theorem for finite semigroups~\cite{S95}.
This theorem and the closely related theorem of Colcombet~\cite{Colcombet07} found several applications in graph theory and combinatorics in recent years, see e.g.\ \cite{BP16, DHM18, BGP21, BP20, NORS21, BP22, PSSTV22, BO24, BO25, CGKKO25, NMPRS21, JMPW24, L25, GGK26, GG26}.
We refer the reader to \cite{B09, Colcombet21} for surveys about Simon's Factorization Theorem.

In our case, the way we apply Simon's Factorization Theorem on path decompositions is directly inspired by the proof method of~\citet{BP16}. The goal of Boja\'nczyk and Pilipczuk was to show that every graph of pathwidth less than $k$ has a tree decomposition of width bounded by a function of $k$ that can be described in Monadic Second Order logic. The description was formalized through the notion of \emph{guided treewidth}, so their result can be concisely stated as follows: the guided treewidth of a graph is bounded by a function of its pathwidth. In fact, this conclusion can be easily derived from the statement of \cref{thm:main_intro}, so our result can be understood as a strengthening of the result of Boja\'nczyk and Pilipczuk; we omit the details.

Let us now briefly explain the approach to the proof of \cref{thm:main_intro}.
Very informally, given a path decomposition of $G$, Simon's Factorization Theorem allows us to  split the path decomposition either into two parts that are arbitrary but simpler, or into an unbounded number of parts that are all simpler and that all behave `in the same way' with respect to the problem at hand.
In the first case, we apply induction on the two simpler parts and combine the resulting tree decompositions, incurring a fixed cost in the width.
In the second case, we first apply induction on each of the parts (which are simpler), and then combine all the resulting tree decompositions in one step, incurring a fixed cost in the width that is independent of the number of parts, thanks to the fact that all the parts behave in the same way. This is the key step of the proof, where we crucially use the power of Simon's Factorization Theorem.

Our application of Simon's Factorization Theorem is described in full details, with a pedagogical aim:  
In our opinion, our application is simpler and less technical than other recent applications in graph theory, and thus it can serve as a good introduction to the technique for the interested reader.

We conclude this introduction by mentioning that the bounding function $f$ in \cref{thm:main_intro} resulting from our proof is quite large (order of $2^{\mathcal{O}(k^2)}$), which is typical of applications of Simon's Factorization Theorem.
It would be interesting to determine whether a polynomial bound exists.
We note that this question was already raised in~\cite{BCHHIM2024}.

\section{Preliminaries}
\label{sec:prelim}

We denote by $\N$ and $\N^+$ the sets of nonnegative and positive integers, respectively.
Given a positive integer $m$, we use the shorthand $[m]$ for the set $\{1, 2, \dots, m\}$.
Given a function $\phi:Y\to Z$ and $X\subseteq Y$, we define $\phi(X)\coloneqq\set{\phi(x)\mid x\in X}$.
All graphs in this paper are finite, simple, and undirected.
Let $G$ be a graph and let $X\subseteq V(G)$.
By $G[X]$ we denote the subgraph of $G$ induced by $X$.
Also, we let $N_G(X)$ denote the set of vertices outside $X$ that have a neighbor in $X$.
If $X=\set{v}$, we write $N_G(v)$ for the set~$N_G(\{v\})$.
For two graphs $G_1 = (V_1, E_1)$ and $G_2 = (V_2, E_2)$, we denote by $G_1 \cup G_2$ the graph $(V_1 \cup V_2, E_1 \cup E_2)$.
For a graph $G = (V, E)$ and a set $F \subseteq \binom{V}{2}$, we denote by $G + F$ the graph $(V, E \cup F)$.

If $G$ is a graph and $A,B,S\subseteq V(G)$, then we say that $S$ {\em{separates}} $A$ and $B$ if every path in $G$ with one endpoint in $A$ and the other endpoint in $B$ intersects $S$.

A \emph{tree decomposition} of a graph $G$ is a pair $(T,(W_x)_{x \in V(T)})$,
where $T$ is a tree and $W_x \subseteq V(G)$ for every $x \in V(T)$, with the
following properties:
\begin{enumerate}
	\item for every vertex $u$ in $G$, the subgraph of $T$ induced by $\{x \in V(T) \mid u \in W_x\}$ is non-empty and connected; and
	\item for every edge $uv$ in $G$, there exists $x\in V(T)$ such that $u,v\in W_x$.
\end{enumerate}

We call the sets $W_x$ the \emph{bags} of the tree decomposition.
The \emph{width} of a tree decomposition is the maximum size of a bag minus one.
The \emph{treewidth} of a graph $G$ is defined to be the minimum width of a tree decomposition of $G$, and is denoted $\tw(G)$.

A \emph{path decomposition} of a graph $G$ is a tree decomposition where the tree is required to be a path.
The minimum width of a path decomposition of $G$ is the {\em pathwidth} of $G$, denoted $\pw(G)$.
In this paper, we denote a path decomposition of $G$ simply as a sequence $(W_1, W_2, \dots, W_m)$
of bags where the order of the bags corresponds to the order of the vertices in the underlying path.
A path decomposition $(W_1, W_2, \dots, W_m)$ is \emph{nice} if $|(W_i \setminus W_{i+1})\cup (W_{i+1}\setminus W_i)| = 1$ for every $i \in [m-1]$.
It is well known (and an easy exercise) that every path decomposition can be turned into a nice path decomposition without increasing its width. In particular, there always exists a minimum-width path decomposition that is nice.

A \emph{semigroup} is a pair $(S,\cdot)$ such that $S$ is a set and $\cdot:S\times S\to S$ is an associative binary operation on $S$, that is, $(s\cdot t)\cdot u = s\cdot (t\cdot u)$ for all $s,t,u\in S$.
An \emph{idempotent} in a semigroup $(S,\cdot)$ is an element $s\in S$ such that $s \cdot s = s$.
Given two semigroups $(S,\cdot)$ and $(T,\circ)$, a \emph{semigroup homomorphism} (from $(S, \cdot)$ to $(T, \circ)$) is a mapping $h:S\to T$ such that for all $s_1,s_2\in S$, we have $h(s_1\cdot s_2) = h(s_1)\circ h(s_2)$.

\subsection{\texorpdfstring{$k$}{k}-interface graphs and abstractions}

We need various definitions and tools that are adapted from \cite{BP16}.

Given a positive integer $k$, a \emph{$k$-interface graph} is a tuple $\mathbb{G}=(G,\phi,L,R)$ where $G$ is a graph, $L, R \subseteq V(G)$ and $\phi:V(G) \to [k]$ is a labeling function that is injective on each of $L$ and $R$.
The vertices in $L$ and $R$ are called {\em left} and {\em right} vertices of $\mathbb{G}$, respectively.

Let $\mathbb{G}_1 = (G_1,\phi_1,L_1,R_1)$ and $\mathbb{G}_2 = (G_2,\phi_2,L_2,R_2)$ be two $k$-interface graphs.
Let $J\coloneqq\phi_1(R_1)\cap \phi_2(L_2) \subseteq [k]$.
We say that $\mathbb{G}_1$ and $\mathbb{G}_2$ are \emph{compatible} if $V(G_1)\cap V(G_2) = \phi_1^{-1}(J) \cap R_1 = \phi_2^{-1}(J) \cap L_2$, and $\phi_1$ and $\phi_2$ coincide on $V(G_1) \cap V(G_2)$.
(Let us emphasize that we do not require that $G_1[V(G_1) \cap V(G_2)]=G_2[V(G_1) \cap V(G_2)]$.)
Observe that if $\mathbb{G}_1$ and $\mathbb{G}_2$ are compatible, then $V(G_1) \cap V(G_2) = R_1 \cap L_2$.

We now define a \emph{gluing} operation on compatible $k$-interface graphs.
Let $\mathbb{G}_1 = (G_1,\phi_1,L_1,R_1)$ and $\mathbb{G}_2 = (G_2,\phi_2,L_2,R_2)$ be two compatible $k$-interface graphs.
Let $G \coloneqq G_1 \cup G_2$, and $\phi(v)\coloneqq \phi_i(v)$ for each $i\in[2]$ and each $v\in V(G_i)$.
(Since $\mathbb{G}_1$ and $\mathbb{G}_2$ are compatible, this is well-defined for $v\in V(G_1)\cap V(G_2)$.)
We define $\mathbb{G}_1\oplus\mathbb{G}_2$ to be the $k$-interface graph $(G,\phi,L_1,R_2)$.

We generalize the notion of compatibility to a sequence $\mathbb{G}_1, \mathbb{G}_2, \ldots, \mathbb{G}_n$ of $k$-interface graphs with $n\geq 2$ as follows: Letting $\mathbb{G}_i = (G_i,\phi_i,L_i,R_i)$ for each $i\in [n]$, we say that the sequence $\mathbb{G}_1, \mathbb{G}_2, \ldots, \mathbb{G}_n$ is {\em compatible} if
\begin{itemize}
	\item $\mathbb{G}_i$ and $\mathbb{G}_{i+1}$ are compatible for each $i\in [n-1]$, and
	\item for each $i, j, \ell \in [n]$ with $i< j<\ell$, if $v\in V(G_i)\cap V(G_\ell)$, then $v\in V(G_j)$.     
\end{itemize}
Observe that if $\mathbb{G}_1, \mathbb{G}_2, \mathbb{G}_3$ is a sequence of three $k$-interface graphs that is compatible, then the two pairs $(\mathbb{G}_1 \oplus \mathbb{G}_2), \mathbb{G}_3$ and $\mathbb{G}_1,  (\mathbb{G}_2 \oplus \mathbb{G}_3)$ are compatible as well. (The easy proof is left to the reader.)
Thus, $(\mathbb{G}_1\oplus\mathbb{G}_2)\oplus\mathbb{G}_3$ and $\mathbb{G}_1\oplus(\mathbb{G}_2\oplus\mathbb{G}_3)$ are well-defined and equal.
It follows from this observation that $\oplus$ is well-defined and associative when considering a sequence $\mathbb{G}_1, \mathbb{G}_2, \ldots, \mathbb{G}_n$ of $k$-interface graphs with $n\geq 2$ that is compatible; in particular, the $k$-interface graph $\mathbb{G}_1 \oplus \mathbb{G}_2 \oplus \cdots \oplus \mathbb{G}_n$ is defined and is independent of the order in which the operation $\oplus$ is carried out.

Two $k$-interface graphs $(G_1,\phi_1,L_1,R_1)$ and $(G_2,\phi_2,L_2,R_2)$  are \emph{isomorphic} (to each other) if there exists a bijective function $\rho: V(G_1)\to V(G_2)$ such that
\begin{enumerate}
	\item for all $u,v\in V(G_1)$,  $uv\in E(G_1) \Leftrightarrow \rho(u)\rho(v)\in E(G_2)$; and
	\item for all $u \in V(G_1)$, we have
	      \begin{itemize}
		      \item $\phi_1(u) = \phi_2(\rho(u))$,
		      \item $u\in L_1 \Leftrightarrow \rho(u) \in L_2$, and
		      \item $u\in R_1 \Leftrightarrow \rho(u) \in R_2$.
	      \end{itemize}
\end{enumerate}
This notion of isomorphism allows us to consider isomorphism classes of $k$-interface graphs.
We let $\mathcal{I}_k$ denote the set of isomorphism classes of $k$-interface graphs.
Given $I\in \mathcal{I}_k$, we say that a $k$-interface graph $\mathbb{G}$ is a {\em representative} for $I$ if its isomorphism class is $I$.

Given $I_1,I_2\in \mathcal{I}_k$, we define $I_1 \oplus I_2$ as follows: 
First, take a representative $\mathbb{G}_i$ of $I_i$ for each $i\in [2]$ in such a way that $\mathbb{G}_1$ and $\mathbb{G}_2$ are compatible.
(Note that this is always possible.\footnote{We remark that this would not be true if we required $G_1[V(G_1) \cap V(G_2)]=G_2[V(G_1) \cap V(G_2)]$ in the definition of gluing.})
Observe that the isomorphism class of $\mathbb{G}_1\oplus\mathbb{G}_2$ does not depend on the choice of $\mathbb{G}_1$ and $\mathbb{G}_2$.
Then, let $I_1 \oplus I_2$ be the isomorphism class of $\mathbb{G}_1\oplus\mathbb{G}_2$.
Observe that $(I_1 \oplus I_2) \oplus I_3= I_1 \oplus (I_2 \oplus I_3)$ holds for every $I_1,I_2,I_3\in \mathcal{I}_k$, thus $\oplus$ is associative on $\mathcal{I}_k$. (This follows from the associativity of $\oplus$ on compatible sequences of $k$-interface graphs.) Therefore, $(\mathcal{I}_k, \oplus)$ is a semigroup.

The following lemma follows easily from the definitions.

\begin{lemma}
	\label{lemma:basic}
	Let $\mathbb{G}_1,\mathbb{G}_2,\dots, \mathbb{G}_n$ be a compatible sequence of $k$-interface graphs, where $n\geq 2$.
	Let $\mathbb{G}\coloneqq  \mathbb{G}_1 \oplus \mathbb{G}_2 \oplus \cdots \oplus \mathbb{G}_n$.
	For each $i\in [n]$, let $\mathbb{G}_i\eqqcolon(G_i, \phi_i, L_i, R_i)$, and let $\mathbb{G}\eqqcolon(G, \phi, L, R)$.
	Then, the following properties hold:
	\begin{enumerate}
		\item $L = L_1$, $R = R_n$, and $\phi_i(u) = \phi(u)$ for each $i\in [n]$ and each $u\in V(G_i)$;
		      \label{basic:item:L-R-phi}
		\item $V(G)=\bigcup_{i\in[n]} V(G_i)$;
		      \label{basic:item:vertex-partition}
		\item $E(G)=\bigcup_{i\in[n]} E(G_i)$;
		      \label{item:basic:surjection-edges}
		\item $R_i \cap L_{i+1}$ separates $\bigcup_{j\in [i]} V(G_j)$ and $\bigcup_{i+1 \leq j \leq n} V(G_j)$ in $G$ for each $i\in [n-1]$.
		      \label{item:basic-separator}
	\end{enumerate}
\end{lemma}

A $k$-interface graph $(G,\phi,L,R)$ is said to be \emph{basic} if $L\cup R = V(G)$.

For a graph $G$ and a set $X$ of vertices of $G$, the \emph{torso} of $G$ with respect to $X$, denoted by $\torso(G,X)$, is the graph with vertex set $X$ where two distinct vertices $x_1,x_2\in X$ are adjacent if and only if they are connected in $G$ by a path whose internal vertices do not belong to $X$. 
(This path may consist of just a single edge, hence vertices of $X$ that are adjacent in $G$ are also adjacent in $\torso(G,X)$.)
The \emph{abstraction}
$\abst{\mathbb{G}}$
of a $k$-interface graph $\mathbb{G}=(G,\phi,L,R)$ is the isomorphism class of the
basic $k$-interface graph $(\torso(G,L\cup R),\phi|_{L\cup R},L,R)$.
Note that if $I \in \mathcal{I}_k$ and $\mathbb{G}, \mathbb{G}'$ are representatives of $I$, then $\abst{\mathbb{G}} = \abst{\mathbb{G}'}$.
Thus, we can define $\abst{I} \coloneqq \abst{\mathbb{G}}$.

Let $\mathcal{A}_k \subseteq \mathcal{I}_k$ be the set of all abstractions of $k$-interface graphs.
Given $A \in \mathcal{A}_k$, there exists a $k$-interface graph $\mathbb{G} = (G, \phi, L, R)$ such that $A = \abst{\mathbb{G}}$. 
Thus, the $k$-interface graph $\mathbb{H} = (\torso(G,L\cup R),\phi|_{L\cup R},L,R)$ is a representative of $A$.
Observe that $\torso(\torso(G, L \cup R), L\cup R)) = \torso(G, L\cup R)$, hence, $\mathbb{H}$ is a representative of $\abst{\mathbb{H}}$, and therefore $\abst{A} = \abst{\mathbb{H}} = A$.

The set $\mathcal{A}_k$ is a finite set of size upper bounded by
$2^k2^k2^k2^{(2k)^2}=2^{\mathcal{O}(k^2)}$: there are $2^k$ possibilities for $\phi(L)\subseteq [k]$, $2^k$ possibilities for $\phi(R)\subseteq [k]$, $2^k$ possibilities for $\phi(L \cap R)\subseteq [k]$, and fewer than $2^{(2k)^2}$ graphs on $|L \cup R| \leq 2k$ vertices.
Given $A_1, A_2\in \mathcal{A}_k$, we define $A_1 \boxplus A_2 \coloneqq \abst{A_1 \oplus A_2}$.
Observe that $(A_1 \boxplus A_2) \boxplus A_3 = \abst{A_1 \oplus A_2 \oplus A_3} = A_1 \boxplus (A_2 \boxplus A_3)$ holds for every $A_1,A_2,A_3\in \mathcal{A}_k$, thus $\boxplus$ is associative on $\mathcal{A}_k$, that is,$(\mathcal{A}_k,\boxplus)$ is a semigroup. 
The proof that $\boxplus$ is associative is elementary but requires unfolding several levels of definitions, and is therefore somewhat technical. 
We defer it to Appendix~\ref{app:boxplus_asso}.

Crucially, observe that $\abst{\cdot} : (\mathcal{I}_k, \oplus) \to (\mathcal{A}_k, \boxplus)$ is a semigroup homomorphism. For the same reason as before, we defer the proof of this statement to Appendix~\ref{app:proof-morphism}.

\subsection{Simon's Factorization Theorem}
Let $(S,\cdot)$ be a finite semigroup.
We denote by $S^+$ the set of all nonempty finite words over $S$.
Given elements $s_1, s_2, \dots, s_n \in S$, we  write $s_1s_2\ldots s_n$
to denote the word in $S^+$ resulting from their concatenation.
For a word $w=s_1s_2\ldots s_n\in S^+$, the length $n$ of $w$ is denoted by $|w|$.
We define the evaluation function $[\cdot]_S:S^+\to S$ as follows:
If $w=s_1s_2\ldots s_n$ with $s_i\in S$ for each $i\in[n]$, then
\[
	[w]_S \coloneqq  s_1 \cdot s_2 \cdots s_n.
\]
A \emph{factorization} of a word $w \in S^+$ is a sequence $(w_1,w_2,\ldots,w_n)$ with $n\geq 1$ such that
$w_i \in S^+$ for each $i\in[n]$, and
$w$ is the concatenation $w_1w_2 \ldots w_n$.
The following two specific types of factorizations will be of special interest for our purposes:
(1) If $n=2$, then the factorization is said to be {\em binary}.
(2) If $n\geq 2$ and $[w_1]_S = \cdots = [w_n]_S$ is an idempotent element of $S$, then
the factorization is said to be {\em unranked}.

The \emph{height} of a word $w \in S^+$ is a positive integer defined by induction on the length of $w$ as follows:
\[
	\hrank(w) = \begin{cases}
		1                                                   & \textrm{if $|w|=1$,}    \\
		1+\min_{(w_1,\ldots,w_n)}\max_{i\in[n]} \hrank(w_i) & \textrm{if $|w|\geq2$,}
	\end{cases}
\]
where the minimum ranges over all factorizations of $w$ that are binary or unranked.

\begin{theorem}[Simon's Factorization Theorem~\cite{S95,K08}] \label{th:simon-factorization}
	Let $(S,\cdot)$ be a finite semigroup.
	Then $\hrank(w) \leq 3|S|$ for every $w\in S^+$.
\end{theorem}

\section{Proof of main result}

\subsection{Forest decompositions}

A \emph{forest decomposition} of a graph $G$ is a pair $(F, (W_x)_{x \in V(F)})$ where $F$ is a forest and $W_x \subseteq V(G)$ for every $x\in V(F)$, with the following properties:
\begin{enumerate}
	\item for every vertex $u$ in $G$, the subgraph of $F$
	      induced by $\{x \in V(F) \mid u \in W_x\}$ is nonempty and connected, and
	      \label{item:forest-decomposition-connected}
	\item for every edge $uv$ in $G$, there exists $x\in V(F)$ such that $u,v\in W_x$.
	      \label{item:forest-decomposition-edges}
\end{enumerate}

We call the sets $W_x$ the \emph{bags} of the forest decomposition.
The \emph{width} of the forest decomposition is the maximum size of a bag, minus one.
Note that if $F$ is connected, then the forest decomposition readily gives a tree decomposition of $G$ with the same width.

A forest decomposition $(F, (W_x)_{x \in V(F)})$ of a graph $G$ is {\em suitable} if it satisfies the following two additional properties:
\begin{enumerate}
	\setcounter{enumi}{2}
	\item $F$ is a subgraph of $G$ and $V(F)=V(G)$, and
	      \label{item:forest-decomposition-spanning}
	\item for every vertex $u \in V(G)$, we have $u \in W_u$.
	      \label{item:forest-decomposition-u-in-its-own-bag}
\end{enumerate}
Observe that every graph has a suitable forest decomposition  $(F, (W_x)_{x \in V(F)})$: let $F$ be an inclusion-wise maximal acyclic subgraph of $G$, and, for each $x\in V(F)$, let the bag $W_x$ be the vertex set of the connected component of $F$ (or, equivalently, of $G$) containing $x$.
We define the {\em complexity of $G$}, denoted $\cmp(G)$, to be the minimum width of a suitable forest decomposition of $G$.
(For definiteness, we let $\cmp(G)\coloneqq 0$ when $G$ has no vertices.)

\begin{lemma}\label{lemma:connected}
Let $G$ be a connected graph and let $(F, (W_x)_{x \in V(F)})$ be a suitable forest decomposition of $G$.
Then, $F$ is connected.
\end{lemma}

\begin{proof}
Suppose for contradiction that $F$ is disconnected.
Let $\{A,B\}$ be a partition of $V(F)$ into two nonempty parts such that $F$ does not contain any $A$--$B$ edges. 
Since $G$ is connected, there exists an edge $uv\in E(G)$ such that $u\in A$ and $v\in B$.
By property~\ref{item:forest-decomposition-edges}, there exists $x\in V(F)$ such that $u,v\in W_x$.
We may assume without loss of generality that $x\in A$.
By property~\ref{item:forest-decomposition-u-in-its-own-bag} of suitable forest decompositions, $v\in W_v$.
However, this implies that the subgraph of $F$ induced by $\{z \in V(F) \mid v \in W_z\}$ contains a vertex from $A$ and a vertex from~$B$; hence, it cannot be connected, contradicting property~\ref{item:forest-decomposition-connected}. 
\end{proof}

\subsection{Intuition and proof sketch}

We first describe a naive approach to a proof of \cref{thm:main_intro}, and explain why it does not work. We then describe how Simon's Factorization Theorem helps us bypass the issues.

Observe first that if $\{A, B\}$ is a partition of $V(G)$ such that both $G[A]$ and $G[B]$ have a suitable forest decomposition of small width, and such that $N_G(B)$ is small, then we can get a suitable forest decomposition of $G$ of small width as follows. 
Start from suitable forest decompositions of small width of $G[A]$ and $G[B]$ and add as many $A$--$B$ edges of $G$ as possible between the two indexing forests so that the resulting graph is still a forest; call it $F$. 
Then, for every $v \in N_G(B)$, add $v$ to all bags $W_x$ such that $x$ is in the path in $F$ between $v$ and a neighbor of $v$ in $B$. 
It is easy to see that this yields a suitable forest decomposition of $G$ of small width. 
This is formalized in \cref{lemma:partition}.

Consider now a graph $G$ with a path decomposition $(W_1, \ldots, W_m)$ of width less than $k$. The following strategy is a natural way to try to apply this observation to obtain a suitable forest decomposition of $G$ of small width. Start from a suitable forest decomposition $F_1$ of $G[W_1]$ of small width $\omega_1$, which exists since $|W_1| \leq k$. 
Then, assuming we have a suitable forest decomposition $F_i$ of width $\omega_i$ of $G[W_1 \cup \cdots \cup W_i]$, using the above observation with ${A = W_1 \cup \cdots \cup W_i}$ and $B = W_{i+1}\setminus \bigcup_{1\le j\le i}W_j$, since $|N_G(B) \cap A| \leq |W_i| \leq k$,
we get a suitable forest decomposition $F_{i+1}$ of width at most $\omega_{i}+k$ of $G[W_1 \cup \cdots \cup W_{i+1}]$.
Observe that this strategy succeeds when $m$ is bounded.

\begin{figure}[h]
    \centering
    \includegraphics[width=\linewidth]{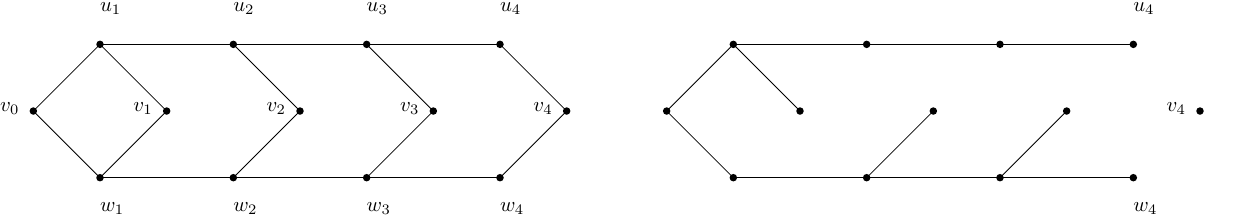}
    \caption{
    Left: The graph $G$ for $n=4$. Right: A possible forest $F_{10}$ produced by the algorithm. Observe that after connecting $v_4$ to $F_{10}$ to create $F_{11}$, either $u_4$ or $w_4$ will have to be added to the bag of $v_0$ in order to be added to the bag of $v_4$.}
    \label{fig:example}
\end{figure}

However, we cannot refine the above analysis to argue that the width of $F_i$ remains bounded, as witnessed by the following example, illustrated in \cref{fig:example}. 
Consider a positive integer $n$ and the graph $G$ with vertex set $\{v_0, \ldots, v_n\} \cup \{u_1, \ldots, u_n\} \cup \{w_1, \ldots, w_n\}$ where $u_1u_2\ldots u_n$ and $w_1w_2\ldots w_n$ form paths, $v_0$ is adjacent to $u_1$ and $w_1$, and $v_i$ is adjacent to $u_i$ and $w_i$ for every $i \in [n]$.
This graph has a minimum-width path decomposition $(W_1, \ldots, W_{3n-1})$ of width $2$ where $W_1 \coloneqq  \set{u_1, v_0, w_1}$, and for every $i \in [n-1]$, $W_{3i-1} \coloneqq  \set{u_i, v_i, w_i}$, $W_{3i} \coloneqq  \set{u_i, u_{i+1}, w_i}$, $W_{3i+1} \coloneqq  \set{u_{i+1}, w_i, w_{i+1}}$; and $W_{3n-1} \coloneqq  \{u_n, v_n, w_n\}$.
With the algorithm described above, for every $i \in [n]$ the forest $F_{3i-1}$ contains the path $u_i, u_{i-1}, \ldots, u_1, v_0, w_1, \ldots, w_{i-1}, w_i$, and every vertex $v_j$ with $j \in [i]$ is adjacent to exactly one vertex, either $u_j$ or $w_j$.
Thus, for every $i \in [n]$, one of the vertices $u_i, w_i$ has to be added to the bag of $v_0$ on its way to being added to the bag of $v_{i}$.
This proves that each forest decomposition $F_{3i-1}$ has width at least $i$, so the final forest decomposition $F_{3n-1}$ has width at least $n$.

Nevertheless, if we could guarantee that in $F_i$, any path between two vertices in $W_i$ stays within $W_i$, 
then when going from $F_i$ to $F_{i+1}$, we would only add vertices to the bags of vertices in $W_i$.
Furthermore, if every vertex appears in at most $\ell$ bags $W_i$ then the size of a given bag $W_x$ would only increase at most $\ell + 1$ times (once when $x$ is added to some $F_i$, and then once every time $x \in W_i$), and each time by at most $k$.
Therefore, the above strategy would work.
In a sense, Simon's Factorization Theorem allows us to obtain such a guarantee that any two vertices at the interface of our current subgraph of $G$ are connected ``locally''.

More precisely, Simon's Factorization Theorem can be used to show that every graph of pathwidth less than $k$ has bounded \emph{height}, where graphs of height $1$ have order at most $2k$, and graphs of height $h+1$ can be built using one of the following two operations: gluing two arbitrary graphs of height $h$ on at most $k$ vertices, or gluing a sequence of graphs $G_1, \ldots, G_n$, each of height $h$, so that this sequence may be arbitrarily long, but possesses a certain ``locality'' property. Informally, this property is that if two vertices from the ``right interface'' after having glued $G_1, \ldots, G_i$ are in the same connected component, then these two vertices were already in the same connected component in $G_i$.

We then prove by induction on $h$ that there exists a function $f(h)$ such that every graph $G$ of height $h$ has a suitable forest decomposition of width at most $f(h)$. 
The base case is simple since graphs of height $1$ have bounded order.
If $G$ has height $h+1$ and is obtained by gluing two arbitrary graphs of height $h$ on at most $k$ vertices then we simply use the first observation to prove that $G$ has a suitable forest decomposition of width at most $f(h)+k$.
If $G$ has height $h+1$ and is obtained by gluing consecutively some arbitrary number of graphs $G_1, \ldots, G_n$ of height $h$ with the ``locality'' property, then we argue that in the above process, the size of a bag only increases at most twice, and each time by at most $k$.
More precisely, we show that we can perform all these steps independently in parallel to build a suitable indexing forest $F$ for $G$ from suitable indexing forests $F_i$ for the $G_i$ (this is Claim~\ref{claim:F-is-a-forest}), and then that we can modify the bags by adding the vertices of the correct interfaces to obtain a suitable forest decomposition of $G$ of width at most $f(h)+3k$ (this is Claim~\ref{claim:suitable-forest-decompostition}).

\subsection{The proof}

We start with some definitions and lemmas.
A {\em near partition} of a set $V$ is a collection of disjoint subsets of $V$ whose union is $V$, with possibly some subsets being empty.

\begin{lemma}
	\label{lemma:partition}
	Let $G$ be a graph and let $\set{A,B}$ be a near partition of $V(G)$.
	Then
	\[
		\cmp(G) \leq \max\set{\cmp(G[A]),\cmp(G[B])} + |N_G(B)|.
	\]
\end{lemma}

\begin{proof}
	If $A$ or $B$ is empty, then $|N_G(B)|=0$ and $\cmp(G) \leq \max\set{\cmp(G[A]),\cmp(G[B])}$
	(recall that the complexity of a graph with no vertex is defined to be $0$).
	Thus, we may assume that $A$ and $B$ are nonempty.

	For each $Z\in \set{A,B}$, let $(F^Z,(W^Z_x)_{x\in Z})$ be a suitable forest decomposition of $G[Z]$ of width $\cmp(G[Z])$.
	Let $M$ be an inclusion-wise maximal subset of $A$--$B$ edges of $G$ such that $F\coloneqq F^A\cup F^B + M$ is a forest.
	For each vertex $x$ of $G$, let $C_x$ be the subset of vertices of
	$N_G(B)$ that are in the same connected component of $F$ as $x$.
	For all $x\in V(G)$, we define
	\[
		W_x \coloneqq \begin{cases}
			W^A_x \cup C_x & \quad\textrm{if $x\in A$,} \\
			W^B_x \cup C_x & \quad\textrm{if $x\in B$.}
		\end{cases}
	\]

	We claim that $(F,(W_x)_{x\in V(F)})$ is a suitable forest decomposition of $G$.
	Let $u$ be a vertex of~$G$, and let $Z\in\set{A,B}$ be such that $u\in Z$.
	Since $(F^Z,(W^Z_x)_{x\in Z})$ is a forest decomposition of~$G[Z]$, we know that
	$F^Z[\set{x \in V(F^Z) \mid u\in W^Z_x}]$ is a nonempty connected subgraph of $F^Z$.
	Now, if $u \not\in N_G(B)$, then
	$\set{x \in V(F) \mid u\in W_x}=\set{x \in V(F^Z)\mid u\in W^Z_x}$ and since $F^Z\subseteq F$, this proves property~\ref{item:forest-decomposition-connected}.
	If $u\in N_G(B)$, let $Y_u$ be the set of vertices that are in the same connected component of $F$ as $u$. Then by construction, we have
	\[\set{x \in V(F)\mid u\in W_x}=\set{x \in V(F^A)\mid u\in W^A_x}\cup Y_u.\]
	Each of these two sets induces a connected subgraph of $F$ containing~$u$.
	Thus, $\set{x\in V(F^A)\mid u\in W^A_x}\cup Y_u$ induces a connected subgraph of $F$, which completes the proof of property~\ref{item:forest-decomposition-connected}.

	In order to prove property~\ref{item:forest-decomposition-edges}, consider an arbitrary edge $uv$ in $G$.
	First assume that $uv$ is an edge of $G[Z]$, for some $Z\in\set{A,B}$.
	Since $(F^Z,(W^Z_x)_{x\in Z})$ is a forest decomposition of $G[Z]$, there is a vertex $x\in Z$ with $u,v \in W^Z_x$. 
    Since $W^Z_x\subseteq W_x$, property~\ref{item:forest-decomposition-edges} holds in this case.
	It remains to consider the case when $u\in A$ and $v\in B$ (up to symmetry).
	This implies that $u\in N_G(B)$ and that
	$u$ and $v$ lie in the same connected component of $G$. By the choice of $M$, it follows that $u$ and $v$ lie in the same connected component of $F$.
	Therefore, $u\in C_v \subseteq W_v$.
	Since also $v\in W_v^A\subseteq W_v$, this completes the proof of property~\ref{item:forest-decomposition-edges}.

	By construction $V(F)=A\cup B = V(G)$ and $F$ is a subgraph of $G$, so property~\ref{item:forest-decomposition-spanning} holds.
	Also, for each $Z\in\set{A,B}$ and $u\in Z$, we have $u\in W^Z_u$
	since $(F^Z,(W^Z_x)_{x\in Z})$ is suitable, and thus $u \in W^Z_u\subseteq W_u$.
	Thus property~\ref{item:forest-decomposition-u-in-its-own-bag} holds as well.
	This completes the proof that $(F,(W_x)_{x\in V(F)})$ is a suitable forest decomposition of $G$.

	Note that $|W_x| \leq \max\set{\cmp(G[A]),\cmp(G[B])} + |N_G(B)| + 1$ for each $x\in V(F)$.
	This completes the proof of the lemma.
\end{proof}

A consequence of the above lemma is that adding one vertex to a graph increases its complexity by at most~$1$. 
More generally:

\begin{lemma}
	\label{lemma:add_vertices}
	Let $G$ be a graph and let $X$ be a subset of vertices of $G$. Then
	\[\cmp(G) \leq \cmp(G-X) + |X|.\]
\end{lemma}

\begin{proof}
	Let $A \coloneqq X$ and $B\coloneqq V(G) - X$.
	Then, $\set{A,B}$ is a near partition of $V(G)$.
	The proof of \cref{lemma:partition} constructs a suitable forest decomposition $(F, (W_x)_{x \in V(F)})$ of $G$ such that $W_x \subseteq A$ for every $x \in A$, and $|W_x| \leq \cmp(G[B]) + |N_G(B)| + 1 \leq \cmp(G-X) + |X| + 1$ for every $x \in B$.
	This readily gives a suitable forest decomposition of $G$ of width at most $\cmp(G-X) + |X|$.
\end{proof}

Before proceeding to our main technical result, we need one more observation regarding $k$-interface graphs and their abstractions.

Let $k$ be a positive integer and let $w\in \mathcal{A}_k^+$; that is, $w=A_1A_2\ldots A_n$ with $A_i\in \mathcal{A}_k$ for each $i\in [n]$.
We define
\begin{align*}
	I_w & \coloneqq  A_1 \oplus  A_2 \oplus \cdots \oplus A_n.
\end{align*}
Note that $I_w \in \mathcal{I}_k$.

\begin{lemma}\label{lem:abst-Iw-is-w}
	Let $k$ be a positive integer and let $w\in \mathcal{A}_k^+$. Then, $\abst{I_w} = [w]_{\mathcal{A}_k}$.
\end{lemma}

\begin{proof}
	Write $w=A_1A_2\ldots A_n$ with $A_i\in \mathcal{A}_k$ for each $i\in [n]$. Then,
	\begin{align*}
		\abst{I_w} & = \abst{A_1 \oplus  A_2 \oplus \cdots \oplus A_n}                                                                                   \\
		           & = \abst{A_1} \boxplus \abst{A_2} \boxplus \cdots \boxplus \abst{A_n} &  & \textrm{since $\abst{\cdot}$ is a semigroup homomorphism} \\
		           & = A_1 \boxplus A_2 \boxplus \cdots \boxplus A_n &  & \textrm{since $\abst{A} = A$ for all $A\in  \mathcal{A}_k$} \\
                   & = [w]_{\mathcal{A}_k}.                                               &  & \qedhere
	\end{align*}
\end{proof}

The following lemma is our main technical tool.

\begin{lemma}
	\label{lem:technical}
	Let $k$ be a positive integer and let $w\in \mathcal{A}_k^+$.
	For each representative $\mathbb{G}=(G, \phi, L, R)$ of $I_w$, and each $X\subseteq L$, we have
	\[
		\cmp(G-X) \leq 3k\cdot\hrank(w) -1.
	\]
\end{lemma}

Before proving \cref{lem:technical}, let us show that it implies our main theorem, which we restate here for convenience.

\restateMain*

\begin{proof}
	We prove the theorem with the function $f:\mathbb{N^+}\to\mathbb{N}$ defined as
	\[
		f(k)\coloneqq  9k\cdot|\mathcal{A}_k|-1.
	\]
	Let $G$ be a connected graph of pathwidth less than $k$ and let $(W_i)_{i\in[m]}$ be a nice path decomposition of $G$ with $|W_i|\leq k$ for each $i\in[m]$.
	Consider a supergraph $G^+$ of $G$ such that $V(G^+)=V(G)$ and $uv$ is an edge in $G^+$ if and only if there is $i\in[m]$ with $u,v\in W_i$.
	Since $(W_i)_{i\in[m]}$ is a path decomposition of $G^+$ such that each bag is a clique, $G^+$ is an interval graph (see~\cite{zbMATH03214396}). 
    Furthermore, the clique number of $G^+$ is bounded by $k$.
	Since interval graphs are perfect, there exists a proper vertex coloring $\phi$ of $G^+$ with at most $k$ colors, say $\phi: V(G)\to[k]$. 
    Observe that $\phi$ is injective on $W_i$ for each $i\in[m]$.

	For each $i\in[m]$, we define
	\[
		G_i \coloneqq G[W_i],\quad
		L_i \coloneqq \begin{cases}
			\emptyset & \textrm{if $i=1$,} \\
			W_i       & \textrm{if $i>1$,}
		\end{cases}
		\quad
		R_i \coloneqq \begin{cases}
			W_i       & \textrm{if $i<m$,} \\
			\emptyset & \textrm{if $i=m$.}
		\end{cases}
	\]
	Note that $\mathbb{G}_i=(G_i,\phi_{|W_i},L_i,R_i)$ is a basic $k$-interface graph for each $i\in[m]$.
	Using the fact that the path decomposition  $(W_i)_{i\in[m]}$ is nice, it can be checked that $\mathbb{G}_i$ and $\mathbb{G}_{i+1}$ are compatible for each $i\in[m-1]$.
	(Let us point out that this could not be guaranteed if the path decomposition were not nice, as then there could be vertices $v\in W_i$ and $w\in W_{i+1}$ with $\phi(v)=\phi(w)$ but $v\neq w$.)
	More generally, it can be checked that the sequence $\mathbb{G}_1, \mathbb{G}_2, \ldots, \mathbb{G}_m$ is compatible.
	Moreover,
	\[
		(G,\phi,\emptyset,\emptyset) = \mathbb{G}_1 \oplus \mathbb{G}_2 \oplus \cdots \oplus \mathbb{G}_m.
	\]
	Let $A_i\coloneqq \abst{\mathbb{G}_i}$ for each $i\in[m]$.
	Consider the word $w\coloneqq A_1A_2\ldots A_m$ in $\mathcal{A}_k^+$.
	Observe that $(G,\phi,\emptyset,\emptyset)$ is a representative for $I_w$.
	Now, by~\cref{lem:technical} (with $X=\emptyset$) and Simon's Factorization Theorem,
	\[
		\cmp(G) \leq 3k\cdot\hrank(w)-1 \leq 3k\cdot 3 |\mathcal{A}_k|-1 = f(k).
	\]
	Therefore, $G$ has a suitable forest decomposition of width at most $f(k)$, which must be a tree decomposition by \cref{lemma:connected} due to the assumption that $G$ is connected.
	This completes the proof of~\cref{thm:main_intro}.
\end{proof}

It remains to prove~\cref{lem:technical}.

\begin{proof}[Proof of~\cref{lem:technical}]
	We prove that the lemma holds for some representative $\mathbb{G}=(G, \phi, L, R)$ of $I_w$ and for every subset $X\subseteq L$.
	By isomorphism of $k$-interface graphs, it will then follow that the lemma holds for every $\mathbb{G}$ and every $X$.

	The proof goes by induction on $\hrank(w)$.
	Suppose first that $\hrank(w)=1$, so $|w|=1$.
	Let $\mathbb{G}=(G, \phi, L, R)$ be a representative of $I_w$ and let $X\subseteq L$.
	Then we have $|V(G)|\leq 2k$.
	In this case, let $F$ be a maximal acyclic subgraph of $G-X$. 
    In particular, $F$ is a forest, and induces a spanning tree on every connected component of $G-X$. For every vertex $x \in V(G - X)$, let $W_x$ be the set of vertices in the same connected component as $x$ in $G - X$ (hence in $F$).
    Then, $(F,(W_x)_{x\in V(F)})$ is a suitable forest decomposition of $G - X$ of width at most $2k-1$.

	Suppose now that $\hrank(w)=h\geq2$.
	Let $(w_1,\ldots,w_n)$ be a factorization of $w$ witnessing that $\hrank(w)=h$.
	The proof splits into two cases.

	Suppose first that the factorization is binary, so $n=2$, $w=w_1w_2$, and $\hrank(w_i)\leq h-1$ for each $i\in[2]$.
	Thus, we have
	\[
		I_w = I_{w_1} \oplus I_{w_2}.
	\]
	Let $\mathbb{G}_1=(G_1,\phi_1,L_1,R_1)$, $\mathbb{G}_2=(G_2,\phi_2,L_2,R_2)$ be compatible representatives of $I_{w_1}$ and  $I_{w_2}$, respectively.
	Let $\mathbb{G}\coloneqq\mathbb{G}_1\oplus \mathbb{G}_2$ and $\mathbb{G}\eqqcolon(G,\phi,L,R$).
	Observe that $\mathbb{G}$ is a representative of $I_w$, and that $L=L_1$ and $R=R_2$ by construction.
	We prove that the lemma holds for this representative. 
    Let thus $X\subseteq L=L_1$, and our goal is to bound $\cmp(G-X)$.

	Define
	\[
		X_1\coloneqq  X, \quad X_2 \coloneqq  R_1 \cap L_2.
	\]
	Since $\mathbb{G}_1$ and $\mathbb{G}_2$ are compatible, we have $V(G_1) \cap V(G_2) = R_1 \cap L_2 = X_2$.
	It follows that
	\[
		\set{V(G_1)-X_1, V(G_2)-X_2}
	\]
	is a near partition of the vertex set of $G-X$ (as $X_1\cap (V(G_2)-X_2)=\emptyset$).

	We claim that
	\begin{equation}
		N_{G-X}(V(G_2) - X_2) \subseteq X_2.
		\label{eq:binary-case-claim}
	\end{equation}
	Consider an arbitrary vertex $u\in N_{G-X}(V(G_2) - X_2)$.
	Let $v$ be a neighbor of $u$ such that $v\in V(G_2) - X_2$.
	Since $v\not\in X_2$, we conclude that $v\not\in V(G_1)\cap V(G_2)$.
	It follows that $v\not\in V(G_1)$.
	However, recall that by~\cref{lemma:basic}.\ref{item:basic:surjection-edges}, $uv \in E(G_1)\cup E(G_2)$.
	Since $v\not\in V(G_1)$, we obtain that $uv\in E(G_2)$
	and in particular $u\in V(G_2)$.
    Furthermore, since \[N_{G-X}(V(G_2) - X_2) \subseteq V(G-X)\setminus (V(G_2) - X_2)\subseteq V(G_1)\,,\] we obtain that $u\in V(G_1)$.
	Thus, $u\in V(G_1) \cap V(G_2) = X_2$.

	By induction,
	\begin{equation}
		\cmp(G_i-X_i) \leq 3k\cdot\hrank(w_i) -1 \leq 3k(h-1) -1
		\label{eq:induction-call}
	\end{equation}
	for each $i\in [2]$.

	Using the above-mentioned near partition $\set{V(G_1)-X_1, V(G_2)-X_2}$  of the vertex set of $G-X$, we obtain
	\begin{align*}
		\cmp(G-X)
		 & \leq \max_{i\in[2]}\cmp(G_i-X_i) + |N_{G-X}(V(G_2)-X_2)| &  & \textrm{by~\cref{lemma:partition}}                                     \\
		 & \leq 3k(h-1) -1 + |X_2|                                                   &  & \textrm{by~\eqref{eq:induction-call} and~\eqref{eq:binary-case-claim}} \\
		 & \leq 3k(h-1) -1 + |L_2|                                                   &  & \textrm{since $X_2\subseteq L_2$}                                      \\
		 & \leq  3k(h-1) -1 + k                                                                                                                                  \\
		 & \leq 3kh -1.
	\end{align*}
	This completes the proof of the case of binary factorization.

	Suppose now that the factorization is unranked, so $n\geq2$, $w=w_1w_2\ldots w_n$, $\hrank(w_i)\leq h-1$
	for each $i\in[n]$, and
	$[w_1]_{\mathcal{A}_k}=\cdots=[w_n]_{\mathcal{A}_k}=A$, where $A\in \mathcal{A}_k$ and $A\boxplus A=A$.
	Note that
	\[
		I_w = I_{w_1} \oplus \cdots \oplus I_{w_n}.
	\]
	Choose a representative $\mathbb{G}_i=(G_i,\phi_i,L_i,R_i)$ of $I_{w_i}$ for each $i\in[n]$ in such a way that the sequence $(\mathbb{G}_1,\ldots, \mathbb{G}_n)$ is compatible.
	Let $\mathbb{G}\coloneqq\mathbb{G}_1\oplus\cdots\oplus \mathbb{G}_n$ and $\mathbb{G}\eqqcolon(G,\phi,L,R)$.
	Observe that $\mathbb{G}$ is a representative of $I_w$, and that by \cref{lemma:basic}.\ref{basic:item:L-R-phi}, we have
	\[
		L = L_1,\quad R = R_n,\quad\textrm{and}\quad \phi|_{V(G_i)} = \phi_i \quad \textrm{for each } i\in[n].
	\]

	For every $i \in [n]$, let $\mathbb{G}'_i \coloneqq \mathbb{G}_1 \oplus \cdots \oplus \mathbb{G}_i$.
	The following claim is the crucial place where we use the idempotence of $A$.

	\begin{claim}\label{cl:all-same-abstraction}
		For every $i \in [n]$, we have $\abst{\mathbb{G}_i} = A$ and $\abst{\mathbb{G}'_i} = A$.
	\end{claim}

	\begin{proof}
		First, for every $i \in [n]$, using \cref{lem:abst-Iw-is-w} for the second equality, we have \[\abst{\mathbb{G}_i} = \abst{I_{w_i}} = [w_i]_{\mathcal{A}_k} = A.\]
		Then, for every $i \in [n]$, we also have \begin{align*}
			\abst{\mathbb{G}'_i} & = \abst{\mathbb{G}_1 \oplus \cdots \oplus \mathbb{G}_i}                                                                                 \\
			                     & = \abst{I_{w_1} \oplus \cdots \oplus I_{w_i}}                          &  & \textrm{by definition of $\abst{\cdot}$ on $\mathcal{I}_k$} \\
			                     & = \abst{I_{w_1}} \boxplus \cdots \boxplus \abst{I_{w_i}}               &  & \textrm{since $\abst{\cdot}$ is a semigroup homomorphism}   \\
			                     & = [w_1]_{\mathcal{A}_k} \boxplus \cdots \boxplus [w_i]_{\mathcal{A}_k} &  & \textrm{by \cref{lem:abst-Iw-is-w}}                         \\
			                     & = A \boxplus \cdots \boxplus A                                                                                                          \\
			                     & = A                                                                    &  & \textrm{by }A\boxplus A=A.\qedhere
		\end{align*}
	\end{proof}

	Write $\mathbb{G}'_i = (G'_i, \phi'_i, L'_i, R'_i)$.
	Informally, we will use \cref{cl:all-same-abstraction} to argue that if there exists a path in $G'_i$ between two vertices of $R'_i = R_i$, whose interior does not intersect $L'_i \cup R'_i = L_1 \cup R_i$, then there also exists a path in $G_i$ between the same two vertices, whose interior does not intersect $L_i \cup R_i$.
	This is the ``locality'' condition we evoked in the proof overview, and which we will use crucially in Claim~\ref{claim:F-is-a-forest}.

	We prove that \cref{lem:technical} holds for the representative $\mathbb{G}$ of $I_w$. 
    Let thus $X\subseteq L_1$, and our goal is to bound $\cmp(G-X)$.
	Here, we follow a slightly different strategy:
	Instead of directly showing that $\cmp(G-X) \leq 3kh-1$, we will show instead that
	\begin{equation}
		\cmp(G-L_1) \leq 3kh - k -1.
		\label{eq:G-L1}
	\end{equation}
	This is enough, because once \eqref{eq:G-L1} is established, it will then follow from \cref{lemma:add_vertices} that
	\[
		\cmp(G-X) \leq \cmp(G-L_1) + |L_1-X| \leq (3kh - k -1) + |L_1-X| \leq 3kh -1.
	\]
	Thus, it remains to prove \eqref{eq:G-L1}, which we do now.

	We start with a claim that relies on the fact that all $\mathbb{G}_i$'s have the same abstraction.
	\begin{claim}
		\label{claim:idempotent-case-persistent-elements}
		Let $i\in[n]$ and let $v$ be a vertex of $G$ such that $v\in L_i\cap R_i$.
		Then $v\in L_j\cap R_j$ for all $j\in[n]$.
	\end{claim}

	\begin{proof}
		Let $\alpha \coloneqq \phi(v)$.
		Since $v \in L_i \cap R_i$, we have $\phi_i^{-1}(\alpha) \cap L_i = \{v\} = \phi_i^{-1}(\alpha) \cap R_i$.
		As all the $\mathbb{G}_j$'s have the same abstraction by \cref{cl:all-same-abstraction}, for every $j \in [n]$ there exists $v_j \in V(G_j)$ such that $\phi_j^{-1}(\alpha) \cap L_j = \{v_j\} = \phi_j^{-1}(\alpha) \cap R_j$. Now, for each $j\in[n-1]$, we have
		$v_j\in R_j$, $v_{j+1}\in L_{j+1}$, and $\phi(v_j)=\alpha=\phi(v_{j+1})$,
		which by the fact that $\mathbb{G}_j$ and $\mathbb{G}_{j+1}$ are compatible, implies
		$v_j=v_{j+1}$.
		Thus, $v_1=v_2=\cdots=v_n=v$.
		This completes the proof of the claim.
	\end{proof}

	Define
	\[
		X_1\coloneqq  L_1\quad \textrm{and}\quad X_{i}\coloneqq  R_{i-1} \cap L_{i}
	\]
	for each $i\in\set{2,\ldots,n}$, and define also
	\[
		V_i\coloneqq  V(G_i)-X_i
	\]
	for each $i\in[n]$.

	\begin{claim}
		\label{claim:structure-of-Vis}
		\hfill
		\begin{enumerate}
			\item $X_i-L_1\subseteq V_{i-1}$, for each $i\in\set{2,\ldots,n}$.
			      \label{item:xi-in-vi-1}
			\item $\set{V_1,\ldots,V_n}$ is a near partition of $V(G)-L_1$.
			      \label{item:vis-form-a-partition}
			\item For every edge $uv$ of $G-L_1$ there is $i\in[n-1]$ such that $u,v\in V_i\cup V_{i+1}$.
			      \label{item:vis-edges}
		\end{enumerate}
	\end{claim}
	\begin{proof}
		For the proof of~\ref{item:xi-in-vi-1},
		let $i\in \set{2,\ldots,n}$ and consider $v\in X_i-L_1$.
		We have $v\in X_i \subseteq R_{i-1}\subseteq V(G_{i-1})$.
		Note that if $v\in L_{i-1}$ then we get a contradiction, as by~\cref{claim:idempotent-case-persistent-elements}, $v\in L_{i-1}\cap R_{i-1}$ implies that $v\in L_1$, which is false.
		Therefore, $v\not\in L_{i-1}$, and in particular $v\not\in X_{i-1}$.
		It follows that $v\in V(G_{i-1})-X_{i-1} = V_{i-1}$.
		This completes the proof of~\ref{item:xi-in-vi-1}.

		For the proof of~\ref{item:vis-form-a-partition}, consider a vertex $v$ of $G-L_1$.
		First, we show that $v \in V_i$ for some $i\in[n]$.
		By~\cref{lemma:basic}.\ref{basic:item:vertex-partition}, there is $i\in[n]$ such that $v\in V(G_i)$. Fix such an index $i$.
		Thus, $v\in V(G_i) = V_i \cup X_i$.
		If $v\in V_i$, then we are done.
		Otherwise, $v\in X_i$.
		Since $v\not\in L_1=X_1$,
		we have $i\geq2$.
		Now by~\ref{item:xi-in-vi-1} we get $v\in V_{i-1}$.
		This completes the proof that $V(G)-L_1\subseteq \bigcup_{j\in[n]} V_j$.
        For the converse inclusion, suppose for  contradiction that there exists $v \in L_1 \cap V_j$ for some $j \in [n]$. 
        Then, $v \in V_j = V(G_j) - X_j$.
        Since $L_1 = X_1$, it holds that $j > 1$.
        Furthermore, $v \in V(G_1)$, since $v \in L_1$. 
        The sequence $\mathbb{G}_1, \ldots, \mathbb{G}_n$ is compatible, so we have $v \in V(G_{j-1})$.
        Thus, $v \in V(G_{j-1}) \cap V(G_j) = R_{j-1} \cap L_j$ since $\mathbb{G}_{j-1}$ and $\mathbb{G}_j$ are compatible. 
        Thus, $v \in X_j$, a contradiction.

		Now we argue that $V_i\cap V_j=\emptyset$ for all distinct $i,j\in[n]$.
		Suppose to the contrary that $v\in V_i\cap V_j$ for vertex $v$ and indices $i, j$, say with $i<j$.
		Recall that by~\cref{lemma:basic}.\ref{item:basic-separator}, $R_{j-1}\cap L_j = X_j$ separates $V_i$ from $V_j$ in $G$.
		This implies that $v\in X_j$, which contradicts the assumption that $v\in V_j$.
		This completes the proof of~\ref{item:vis-form-a-partition}.

		For the proof of~\ref{item:vis-edges}, consider an arbitrary edge $uv$ of $G-L_1$.
		By~\cref{lemma:basic}.\ref{item:basic:surjection-edges}, we can fix
		$i\in[n]$ such that $u,v\in V(G_i)$.
		Thus, $u,v\in V_i\cup (X_i-L_1)$.
		If $i=1$, then we are done since $X_1=L_1$.
		If $i\geq2$, then by~\ref{item:xi-in-vi-1} we have $X_i-L_1\subseteq V_{i-1}$, so we conclude that $u,v \subseteq V_{i-1}\cup V_i$.
		This concludes the proof of~\ref{item:vis-edges}.
	\end{proof}

	Since $\hrank(w_i)\leq h-1$ for each $i\in [n]$, by induction
	\[
		\cmp(G_i-X_i) \leq 3k(h-1)-1
	\]
	for each $i\in[n]$.
	For each $i\in[n]$, let $(F_i,\set{W_{i,x}}_{x\in V_i})$ be a suitable forest decomposition of $G_i-X_i$ witnessing $\cmp(G_i-X_i) \leq 3k(h-1)-1$.

	For each $i\in[n-1]$, let $M_i$ be an inclusion-wise maximal subset of $V_i$--$V_{i+1}$ edges in $G-L_1$ such that
	$F_i\cup F_{i+1} + M_i$ is a forest.
	Let
	\[
		F = \bigcup_{i\in[n]} F_i + \bigcup_{i\in[n-1]} M_i.
	\]
	\begin{claim}
		\label{claim:F-is-a-forest}
		$F$ is a forest and $V(F)=V(G)-L_1$.
	\end{claim}
	\begin{proof}
		First recall that
		$V(F_i)=V(G_i)-X_i=V_i$ for each $i\in[n]$, and that by~\cref{claim:structure-of-Vis}.\ref{item:vis-form-a-partition},
		the set $\set{V_1,\ldots,V_n}$ is a near partition of $V(G)-L_1$.
        Thus, $V(F)=V(G)-L_1$, and it only remains to show that $F$ is a forest.

		Arguing by contradiction, suppose that $F$ contains a cycle.
		For each cycle $C$ in $F$ define the {\em index $i(C)$ of $C$} to be the largest $i \in [n]$ such that
		$V(C)\cap V_i\neq\emptyset$.
		Now, let $C$ be a cycle in $F$ minimizing $i(C)$.
		For convenience, let $i\coloneqq i(C)$.
		Recall that $F[V_i]=F_i$ is a forest, thus
		$i\geq2$ and, by \cref{claim:structure-of-Vis}.\ref{item:vis-edges}, $V(C)\cap V_{i-1}\neq\emptyset$.
		Recall also that by the definition of $M_{i-1}$ we have that
		$F[V_{i-1}\cup V_i] = F_{i-1}\cup F_i + M_{i-1}$ is a forest.
		It follows that $i\geq3$ and that
		\begin{equation}
			\label{eq:P_intersects_i-2}
			V(C)\cap \bigcup_{j \in [i-2]} V_j\neq\emptyset.
		\end{equation}

		We claim that there exists a subpath $P$ of $C$ with endpoints $u$ and $v$ such that
		\begin{enumerate}
			\item $V(P)$ is disjoint from $V_i$;
			\item $P$ is an $R_{i-1}$-path in $G-L_1$, that is,
			      $u,v\in R_{i-1}$ and no other vertex of $P$ is in $R_{i-1}$;
			\item $V(P)$ intersects $\bigcup_{j \in [i-2]} V_j$.\label{item:intersects}
		\end{enumerate}
        To see this, observe first that, since $C$ contains a vertex of $\bigcup_{j \in [i-2]} V_j$ (by \eqref{eq:P_intersects_i-2}),
        there exists a subpath of $C$ that is disjoint from $V_i$ and intersects $\bigcup_{j \in [i-2]} V_j$.
		Let $Q$ be a an inclusion-wise maximal such subpath, and let $u$ and $v$ be the endpoints of $Q$. 
		Since $C$ contains a vertex of $V_i$ and $u,v\not\in V_i$, we conclude that both $u$ and $v$ have neighbors in $V_i$.
		Thus, we conclude that $u$ and $v$ belong to the separator~$X_i$, and therefore $u,v \in X_i \subseteq R_{i-1}$.
		Thus, $Q$ satisfies all the properties listed above except possibly for the fact that $Q$ could have some internal vertices in $R_{i-1}$. In the latter case, there is a subpath of $Q$ that has the desired properties.

		Consider $\mathbb{G}' \coloneqq \mathbb{G}'_{i-1} (= \mathbb{G}_1\oplus\cdots\oplus \mathbb{G}_{i-1})$ and
		let $\mathbb{G}' \eqqcolon (G',\phi',L',R')$.
		Observe that by \cref{lemma:basic}.\ref{basic:item:L-R-phi}, we have
		\[
			L' = L_1,\quad R' = R_{i-1},\quad\textrm{and}\quad \phi' = \phi_{|V(G')}.
		\]
		By \cref{cl:all-same-abstraction}, $A$ is the abstraction of both
		$\mathbb{G}'$ and $\mathbb{G}_{i-1}$.
		Hence, the two basic $k$-interface graphs
		\[(\torso(G',L_1,R_{i-1}),\phi|_{L_1\cup R_{i-1}},L_1,R_{i-1})\textrm{ and }
			(\torso(G_{i-1},L_{i-1},R_{i-1}),\phi|_{L_{i-1}\cup R_{i-1}},L_{i-1},R_{i-1})\] are isomorphic.

		Recall that $P$ is a path in $G[\bigcup_{j\in [i-1]} V_j]$ with both endpoints in $R_{i-1}$ and no
		internal vertex in $R_{i-1}$.
		Moreover, $P$ avoids $L_1$ since cycle $C$ avoids $L_1$.
		Note also that $G[\bigcup_{j\in [i-1]} V_j] - L_1 = G'-L_1$.
		We deduce that $P$ has no internal vertex in $L_1 \cup R_{i-1}$, and hence
		$uv$ is an edge in $\torso(G',L_1,R_{i-1})$.
		By the isomorphism mentioned above,
		$\torso(G_{i-1},L_{i-1},R_{i-1})$ contains an edge $u'v'$ with
		$u',v'\in R_{i-1}$,
		$\phi'(u')=\phi_{i-1}(u)$, and $\phi'(v')=\phi_{i-1}(v)$.
		Recall that $\phi' = \phi|_{V(G')}$, $\phi_{i-1} = \phi|_{V(G_{i-1})}$, and that $\phi'$ is injective on $R_{i-1}$ (as this is the set of right vertices of $\mathbb{G}_{i-1}$).
		This implies that $u'=u$ and $v'=v$, thus $\torso(G_{i-1},L_{i-1},R_{i-1})$ contains the edge $uv$.
		This edge is witnessed by a $uv$-path $P'$ in $G_{i-1}$ such that $P'$ contains no internal vertex in $L_{i-1}\cup R_{i-1}$. In particular, $P'$ is a path in $G_{i-1}-X_{i-1}$, and so $V(P')\subseteq V_{i-1}$.

		Recall that $P$ and $P'$ both avoid $V_i$.
		Furthermore, $V(P)$ intersects $\bigcup_{j\in [i-2]} V_j$ by \ref{item:intersects},
		while $P'$ does not, implying that $P' \neq P$.
		It follows that $P\cup P'$ contains a cycle $C'$.
		Note that $i(C') = i-1 < i=i(C)$, since $C'$ avoids $V_i$.
		This shows that $C'$ should have been chosen instead of $C$, a contradiction.
		This concludes the proof of the claim.
	\end{proof}

	Using \cref{claim:F-is-a-forest}, we now build a suitable forest decomposition of $G-L_1$ as follows.
	For each $i\in[n-1]$ and each vertex $x \in V_i \cup V_{i+1}$, let $C_{i,i+1}(x)$ be the connected component of $F_i\cup F_{i+1} + M_i$ containing $x$.
	For each $i\in[n]$ and each $x\in V_i$, we define
	\[
		W_x \coloneqq  \begin{cases}
			W_{1,x} \cup (V(C_{1,2}(x))\cap R_{1})                                      & \textrm{if $i=1$,}                    \\
			W_{i,x} \cup (V(C_{i-1,i}(x))\cap R_{i-1}) \cup (V(C_{i,i+1}(x))\cap R_{i}) & \textrm{if $i\in\set{2,\ldots,n-1}$,} \\
			W_{n,x} \cup (V(C_{n-1,n}(x))\cap R_{n-1})                                  & \textrm{if $i=n$.}
		\end{cases}
	\]
	We call the vertices in $W_x - W_{i,x}$ the {\em new vertices} of $W_x$.
	Observe that each new vertex of $W_x$ belongs to either $R_{i-1}$ or $R_i$.

	\begin{claim}
		$(F,\set{W_x}_{x\in V(G)-L_1})$ is a suitable forest decomposition of $G-L_1$.
		\label{claim:suitable-forest-decompostition}
	\end{claim}
	\begin{proof}
		To prove the claim, let us show that properties \ref{item:forest-decomposition-connected}, \ref{item:forest-decomposition-edges}, \ref{item:forest-decomposition-spanning}, and \ref{item:forest-decomposition-u-in-its-own-bag} from the definition of a suitable forest decomposition all hold.
		By~\cref{claim:F-is-a-forest}, $F$ is a forest contained in $G-L_1$ with $V(F)=V(G)-L_1$, thus \ref{item:forest-decomposition-spanning} holds.
		Also, $u\in W_u$ for every $u\in V(G)-L_1$ by construction, so \ref{item:forest-decomposition-u-in-its-own-bag} holds as well.

		Next, we show property \ref{item:forest-decomposition-connected}, stating that for every $u\in V(G)-L_1$,
		the subgraph of $F$ induced by $\set{x \in V(F) \mid u \in W_x}$ is nonempty and connected.
		Let $u\in V(G)-L_1$ and let $i$ be such that $u\in V_i$ (which exists by \cref{claim:structure-of-Vis}.\ref{item:vis-form-a-partition}).
		Suppose first that $u\not\in R_i$ or $i=n$.
		Then, $u$ is not a new vertex of any bag $W_x$ with $x\in V(F)$, and it follows that
		\[
			\set{x\in V(F) \mid u\in W_x} = \set{x\in V(F_i) \mid u\in W_{i,x}}.
		\]
		Since the latter set is nonempty and connected in $F_i$, and since $F_i\subseteq F$, this set also induces a nonempty and connected subgraph of $F$, as desired.

		Suppose now that $i\in[n-1]$ and $u\in R_i$.
		Here, $u$ is potentially a new vertex of some bags $W_x$ with $x\in V(F)$.
		Observe that in this case
		\[
			\set{x\in V(F) \mid u\in W_x} = \set{x\in V(F_i) \mid u\in W_{i,x}}
			\cup V(C_{i,i+1}(u)).
		\]
		Note that $\set{x\in V(F_i) \mid u\in W_{i,x}}$ contains $u$ and induces a connected subgraph of $F_i$.
		Recall also that $C_{i,i+1}(u)$ is the connected component of $F_i\cup F_{i+1} + M_i$ containing $u$.
		Since $F_i\cup F_{i+1} + M_i\subseteq F$, it follows that $\set{x\in V(F) \mid u\in W_x}$ induces a nonempty and connected subgraph of $F$, as desired.
		Therefore, property \ref{item:forest-decomposition-connected} holds.

		It remains to show property \ref{item:forest-decomposition-edges}, stating that for every edge $uv$ in $G-L_1$ there is $x\in V(F)$ such that $u,v\in W_x$.
		Let $uv$ be an edge of $G-L_1$, and let $i\in[n]$ be such that $uv\in E(G_i)$ (which exists by~\cref{lemma:basic}.\ref{item:basic:surjection-edges}).
		Then, either (1) $i \geq 2$, $u, v \in X_i$, or (2) $u,v\in V_i$, or (3) $i\geq 2$, $u\in X_i$, and $v\in V_i$ (possibly exchanging $u$ and $v$ first if necessary).
		Suppose first that $u, v \in X_i$. Since $u, v \notin L_1$ then $u, v \in V_{i-1}$ by \cref{claim:structure-of-Vis}.\ref{item:xi-in-vi-1}. Thus, up to replacing $i$ by $i-1$, this case reduces to the case where $u, v \in V_i$.
		Suppose now that $u,v\in V_i$.
		Since $(F_i,\set{W_{i,x}}_{x\in V_i})$ is a forest decomposition of $G_i-X_i$,	there exists $x\in V_i$ such that $u,v \in W_{i,x} \subseteq W_x$, as desired.

		Suppose finally that $i\geq2$, $u\in X_i$ and $v\in V_i$.
		Then, it follows from \cref{claim:structure-of-Vis}.\ref{item:xi-in-vi-1} that $u\in V_{i-1}$.
		Thus, $uv$ is a $V_{i-1}$--$V_i$ edge.
		By the choice of $M_{i-1}$, this implies that $u$ and $v$ are in the same connected component of $F_{i-1} \cup F_i + M_{i-1}$, namely $C_{i-1,i}(u)=C_{i-1,i}(v)$.
		It follows from the definition of the bags that $V(C_{i-1,i}(v)) \cap R_{i-1} \subseteq W_v$.
		Since $u\in R_{i-1}$, this implies that $u \in W_v$. Since $v\in W_v$, we have $u,v \in W_v$, as desired.
		Hence property \ref{item:forest-decomposition-edges} holds.

		Therefore, $(F,\set{W_x}_{x\in V(G)-X})$ is a suitable forest decomposition of $G-L_1$, as claimed.
	\end{proof}

	Now, it only remains to bound the width of $(F,\set{W_x}_{x\in V(G)-X})$. Since $|W_{i, x}| \leq 3k(h-1)$ holds for every $i\in [n]$ and $x\in V_i$ by induction, and since $|R_{i}| \leq k$ holds for every $i\in [n]$, we deduce that $|W_x| \leq 3k(h-1) + 2k$ holds for every $x\in V(G-L_1)$, as desired.
	This concludes the proof of \eqref{eq:G-L1}, and of the lemma.
\end{proof}

\section*{Acknowledgments}

This research was carried out during the second edition of the Structural Graph Theory Workshop (STWOR) held at the conference center of the University of Warsaw in Chęciny (Poland), June 30th--July 5th, 2024. 
The workshop was supported by the Excellence Initiative -- Research University (IDUB) funds of the University of Warsaw, as well as project BOBR that is funded from the European Research Council (ERC) under the European Union’s Horizon 2020 research and innovation programme with grant agreement No.~948057. 
We thank the workshop organizers and other participants for providing a stimulating working environment.
We are particularly grateful to Linda Cook for suggesting us to work on the topic of this paper and for her contributions in the research discussions.

\bibliographystyle{abbrvnat}
\bibliography{bibliography}

\appendix

\section{\texorpdfstring{$\boxplus$}{+} is associative} \label{app:boxplus_asso}

Let $A_1,A_2,A_3\in \mathcal{A}_k$. We want to show that $(A_1 \boxplus A_2) \boxplus A_3 = \abst{A_1 \oplus A_2 \oplus A_3} = A_1 \boxplus (A_2 \boxplus A_3).$
By symmetry, it suffices to prove that \[(A_1 \boxplus A_2) \boxplus A_3 = \abst{A_1 \oplus A_2 \oplus A_3}.\]
For every $i \in [3]$, let $\mathbb{G}_i$ be a representative of $A_i$, chosen so that the sequence $\mathbb{G}_1, \mathbb{G}_2, \mathbb{G}_3$ is compatible.
Let $\mathbb{G}_i \eqqcolon (G_i, \phi_i, L_i, R_i)$ for every $i \in [3]$.
Define $\phi(v) \coloneqq \phi_i(v)$ for each $i \in [3]$ and $v \in V(G_i)$, which is well-defined since the sequence $\mathbb{G}_1, \mathbb{G}_2, \mathbb{G}_3$ is compatible.

By definition, we have $(A_1 \boxplus A_2) \boxplus A_3 = \abst{(A_1 \boxplus A_2) \oplus A_3} = \abst{\abst{A_1 \oplus A_2} \oplus A_3}$.
Note that $\abst{A_1 \oplus A_2} = \abst{\mathbb{G}_1 \oplus \mathbb{G}_2}$ by definition.
Recall that $\mathbb{G}_1 \oplus \mathbb{G}_2 = (G_1 \cup G_2, \phi|_{V(G_1) \cup V(G_2)}, L_1, R_2)$.
Therefore, the $k$-interface graph $\mathbb{T} = (\torso(G_1 \cup G_2, L_1 \cup R_2), \phi|_{L_1 \cup R_2}, L_1, R_2)$ is a representative of $\abst{\mathbb{G}_1 \oplus \mathbb{G}_2} = \abst{A_1 \oplus A_2}$.
Since the sequence $\mathbb{G}_1, \mathbb{G}_2, \mathbb{G}_3$ is compatible, $\mathbb{T}$ and $\mathbb{G}_3$ are compatible.
Thus, $\mathbb{T} \oplus \mathbb{G}_3$ is a representative of $\abst{A_1 \oplus A_2} \oplus A_3$.
By definition, we then have $(A_1 \boxplus A_2) \boxplus A_3 = \abst{\abst{A_1 \oplus A_2} \oplus A_3} = \abst{\mathbb{T} \oplus \mathbb{G}_3}$.
Let $T \coloneqq \torso(G_1 \cup G_2, L_1 \cup R_2)$.
Then, $\mathbb{T} \oplus \mathbb{G}_3 = (T \cup G_3, \phi|_{V(T) \cup V(G_3)}, L_1, R_3)$ and hence, the $k$-interface graph $\mathbb{H}_1 \coloneqq (\torso(T \cup G_3, L_1 \cup R_3), \phi|_{L_1 \cup R_3}, L_1, R_3)$ is a representative of $\abst{\mathbb{T} \oplus \mathbb{G}_3} = (A_1 \boxplus A_2) \boxplus A_3$.

Since $\mathbb{G}_1 \oplus \mathbb{G}_2 \oplus \mathbb{G}_3 = (G_1 \cup G_2 \cup G_3, \phi, L_1, R_3)$ is a representative of $A_1 \oplus A_2 \oplus A_3$, we have that $\abst{A_1 \oplus A_2 \oplus A_3} = \abst{\mathbb{G}_1 \oplus \mathbb{G}_2 \oplus \mathbb{G}_3}$.
Thus, the $k$-interface graph $\mathbb{H}_2 \coloneqq (\torso(G_1 \cup G_2 \cup G_3, L_1 \cup R_3), \phi|_{L_1 \cup R_3}, L_1, R_3)$ is a representative of $\abst{\mathbb{G}_1 \oplus \mathbb{G}_2 \oplus \mathbb{G}_3} = \abst{A_1 \oplus A_2 \oplus A_3}$. 

Therefore, to prove that $(A_1 \boxplus A_2) \boxplus A_3 = \abst{A_1 \oplus A_2 \oplus A_3}$, it suffices to prove that the $k$-interface graphs $\mathbb{H}_1$ and $\mathbb{H}_2$ are equal. 
To do so, it suffices to prove that the graphs $\torso(T \cup G_3, L_1 \cup R_3)$ and $\torso(G_1 \cup G_2 \cup G_3, L_1 \cup R_3)$ are equal. 
Since both graphs have the same vertex set, namely $L_1 \cup R_3$, it suffices to prove that they have the same edges.

Let $u, v \in L_1 \cup R_3$ and suppose that $uv \in E(\torso(G_1 \cup G_2 \cup G_3, L_1 \cup R_3))$. 
Thus, there exists a path $P$ from $u$ to $v$ in $G_1 \cup G_2 \cup G_3$ that is internally disjoint from $L_1 \cup R_3$.
Write $P = w_0 P_0 w_1 P_1 w_2 \ldots w_\ell P_{\ell} w_{\ell+1}$ with $w_0 = u$ and $w_{\ell+1} = v$, where each $w_i$ is in $L_1 \cup R_2 \cup L_3 \cup R_3$, and each $P_i$ is internally disjoint from $L_1 \cup R_2 \cup L_3 \cup R_3$.
Since $R_2 \cap L_3$ separates $V(G_1) \cup V(G_2)$ and $V(G_3)$ in $G_1 \cup G_2 \cup G_3$ by \cref{lemma:basic}.\ref{item:basic-separator}, each $P_i$ is either entirely contained in $G_1 \cup G_2$ or entirely contained in $G_3$.
Thus, each $P_i$ is either a path in $G_1 \cup G_2$ that is internally disjoint from $L_1 \cup R_2$, or a path in $G_3$ that is internally disjoint from $L_3 \cup R_3$.
In the first case, observe that $P_i$ witnesses an edge $w_iw_{i+1}$ in $T = \torso(G_1 \cup G_2, L_1 \cup R_2)$.
Therefore, replacing each such path $P_i$ by the edge $w_iw_{i+1}$ of $T$ yields a path $P'$ from $u$ to $v$ in $T \cup G_3$ that is internally disjoint from $L_1 \cup R_3$ (as $V(P') \subseteq V(P)$).
This proves that $uv \in E(\torso(T \cup G_3, L_1\cup R_3))$.

Conversely, suppose that $uv \in E(\torso(T \cup G_3, L_1\cup R_3))$.
Then, there exists a path $P'$ from $u$ to $v$ in $T \cup G_3$ that is internally disjoint from $L_1 \cup R_3$.
Write $P' = w_0P'_0w_1P'_1\ldots w_{\ell}P'_{\ell}w_{\ell+1}$ with $w_0 = u$ and $w_{\ell+1} = v$, where each $w_i$ is in $L_1 \cup R_2 \cup L_3 \cup R_3$ and each $P'_i$ is internally disjoint from $L_1 \cup R_2 \cup L_3 \cup R_3$.
Observe that each $P'_i$ is either an edge of $T$ or a path in $G_3$ that is internally disjoint from $L_1 \cup R_2 \cup L_3 \cup R_3$.
For every $i$ such that $P'_i$ is an edge $w_iw_{i+1}$ of $T = \torso(G_1 \cup G_2, L_1 \cup R_2)$, there exists a path $P_i$ from $w_i$ to $w_{i+1}$ in $G_1 \cup G_2$ that is internally disjoint from $L_1 \cup R_2$, hence, internally disjoint from $L_1 \cup R_3$, as the sequence $\mathbb{G}_1, \mathbb{G}_2, \mathbb{G}_3$ is compatible.
Therefore, replacing each such path $P'_i$ by the path $P_i$ yields a walk that contains a path $P$ from $u$ to $v$ in $G_1 \cup G_2 \cup G_3$ that is internally disjoint from $L_1 \cup R_3$.
This proves that $uv \in E(\torso(G_1 \cup G_2 \cup G_3, L_1\cup R_3))$.
Therefore, the graphs $(\torso(T \cup G_3, L_1 \cup R_3)$ and $\torso(G_1 \cup G_2 \cup G_3, L_1 \cup R_3)$ are equal, which concludes the proof.

\section{\texorpdfstring{$\abst{\cdot}$}{[[.]]} is a semigroup homomorphism} \label{app:proof-morphism}

Let $I_1, I_2 \in \mathcal{I}_k$. 
We want to show that $\abst{I_1 \oplus I_2} = \abst{I_1} \boxplus \abst{I_2}$.

For every $i \in [2]$, let $\mathbb{G}_i$ be a representative of $I_i$, chosen so that $\mathbb{G}_1$ and $\mathbb{G}_2$ are compatible.
Fix $i \in [2]$. 
We define $\mathbb{G}_i \eqqcolon (G_i, \phi_i, L_i, R_i)$, $H_i \coloneqq \torso(G_i,L_i\cup R_i)$, and $\mathbb{H}_i \coloneqq (H_i,\phi_i|_{L_i\cup R_i},L_i,R_i)$.
By definition, $\mathbb{H}_i$ is a representative of $\abst{\mathbb{G}_i}$.
Thus, by definition, $\mathbb{H}_i$ is a representative of $\abst{I_i}$.
Therefore, $\mathbb{H}_1 \oplus \mathbb{H}_2$ is a representative of $\abst{I_1} \oplus \abst{I_2}$.

Therefore, again by definition, we have \begin{align*}
    \abst{I_1} \boxplus \abst{I_2} &= \abst{\abst{I_1} \oplus \abst{I_2}} &  & \textrm{by definition of $\boxplus$} \\
        &= \abst{\mathbb{H}_1 \oplus \mathbb{H}_2} & & \textrm{by definition of $\abst{\cdot}$ on $\mathcal{I}_k$}
\end{align*}

By definition, $\mathbb{G}_1 \oplus \mathbb{G}_2$ is a representative of $I_1 \oplus I_2$, so $\abst{I_1 \oplus I_2} = \abst{\mathbb{G}_1 \oplus \mathbb{G}_2}$, again by definition.
Thus, to prove that $\abst{I_1 \oplus I_2} = \abst{I_1} \boxplus \abst{I_2}$, it suffices to prove that $\abst{\mathbb{G}_1 \oplus \mathbb{G}_2} = \abst{\mathbb{H}_1 \oplus \mathbb{H}_2}$.

Let $G \coloneqq G_1 \cup G_2$, and define $\phi(v) \coloneqq \phi_i(v)$ for each $i \in [2]$ and $v \in V(G_i)$, which is well-defined since $\mathbb{G}_1$ and $\mathbb{G}_2$ are compatible.
Then, $\mathbb{G}_1 \oplus \mathbb{G}_2 = (G, \phi, L_1, R_2)$.
Thus, $\abst{\mathbb{G}_1 \oplus \mathbb{G}_2}$ is the isomorphism class of the $k$-interface graph $(\torso(G, L_1\cup R_2), \phi|_{L_1 \cup R_2}, L_1, R_2)$.

Let $H \coloneqq H_1 \cup H_2$. 
Then, $\mathbb{H}_1 \oplus \mathbb{H}_2 = (H, \phi|_{L_1 \cup R_1 \cup L_2 \cup R_2}, L_1, R_2)$.
Thus, $\abst{\mathbb{H}_1 \oplus \mathbb{H}_2}$ is the isomorphism class of the $k$-interface graph $(\torso(H, L_1\cup R_2), \phi|_{L_1 \cup R_2}, L_1, R_2)$.
Therefore, to prove that $\abst{\mathbb{G}_1 \oplus \mathbb{G}_2} = \abst{\mathbb{H}_1 \oplus \mathbb{H}_2}$, it suffices to prove that the $k$-interface graphs $(\torso(G, L_1\cup R_2), \phi|_{L_1 \cup R_2}, L_1, R_2)$ and $(\torso(H, L_1\cup R_2), \phi|_{L_1 \cup R_2}, L_1, R_2)$ are equal, or, equivalently, to prove that the graphs $\torso(G, L_1\cup R_2)$ and $\torso(H, L_1\cup R_2)$ are equal. 
Since both graphs have the same vertex set, namely $L_1 \cup R_2$, it suffices to prove that they have the same edges.

Let $u, v \in L_1 \cup R_2$ and suppose that $uv \in E(\torso(G, L_1\cup R_2))$. Thus, there exists a path $P$ from $u$ to $v$ in $G$ that is internally disjoint from $L_1 \cup R_2$.
Write $P = w_0 P_0 w_1 P_1 w_2 \ldots w_\ell P_{\ell} w_{\ell+1}$ with $w_0 = u$ and $w_{\ell+1} = v$, where each $w_i$ is in $L_1 \cup R_1 \cup L_2 \cup R_2$, and each $P_i$ is internally disjoint from $L_1 \cup R_1 \cup L_2 \cup R_2$.
Since $R_1 \cap L_2$ separates $V(G_1)$ and $V(G_2)$ in $G = G_1 \cup G_2$ by \cref{lemma:basic}.\ref{item:basic-separator}, each $P_i$ is either entirely contained in $G_1$ or entirely contained in $G_2$.
Thus, each $P_i$ is either a path in $G_1$ that is internally disjoint from $L_1 \cup R_1$, or a path in $G_2$ that is internally disjoint from $L_2 \cup R_2$.
Thus, for every $i \in \{0, \ldots, \ell\}$, there is an edge $w_iw_{i+1}$ either in $H_1 = \torso(G_1, L_1 \cup R_1)$ or in $H_2 = \torso(G_2, L_2 \cup R_2)$.
Therefore, $P' \coloneqq uw_1w_2\ldots w_lv$ is a path from $u$ to $v$ in $H = H_1 \cup H_2$ that is internally disjoint from $L_1 \cup R_2$.
This proves that $uv \in E(\torso(H, L_1\cup R_2))$.

Conversely, suppose that $uv \in E(\torso(H, L_1\cup R_2))$.
Then, there exists a path $P'$ from $u$ to $v$ in $H$ that is internally disjoint from $L_1 \cup R_2$.
Write $P' = w_0w_1w_2\ldots w_{\ell}w_{\ell+1}$, with $w_0 = u$ and $w_{\ell+1} = v$.
For every $i \in \{0, \ldots, \ell\}$, $w_iw_{i+1}$ is an edge of $H = H_1 \cup H_2$, so there exists a path $P_i$ from $w_i$ to $w_{i+1}$, either in $G_1$ and internally disjoint from $L_1 \cup R_1$, or in $G_2$ and internally disjoint from $L_2 \cup R_2$.
Then, each $P_i$ is a path in $G_1 \cup G_2$ that is internally disjoint from $L_1 \cup R_1 \cup L_2 \cup R_2$ (because $\mathbb{G}_1$ and $\mathbb{G}_2$ are compatible).
Thus, $w_0P_0w_1\ldots w_{\ell}P_{\ell}w_{\ell+1}$ is a walk from $u$ to $v$ in $G = G_1 \cup G_2$ that contains a path from $u$ to $v$ in $G$ that is internally disjoint from $L_1 \cup R_2$.
This proves that $uv \in E(\torso(G, L_1\cup R_2))$.
Therefore, the graphs $\torso(G, L_1\cup R_2)$ and $\torso(H, L_1\cup R_2)$ are equal, which concludes the proof.
\end{document}